\documentclass[11pt,a4paper]{article}

\usepackage{pgf,tikz}
\usepackage{physics}
\usepackage{mathrsfs}
\usetikzlibrary{arrows}
\usetikzlibrary{arrows.meta}
\usepackage[T1]{fontenc}
\usepackage{amsmath}
\usepackage{amsfonts}
\usepackage{amssymb}
\usepackage{amsthm}
\usepackage{amscd}
\usepackage{graphicx}
\usepackage{enumitem}

\usepackage{authblk}

\usepackage{verbatim}
\usepackage{tikz, tkz-graph}
\usepackage{hyperref}

\usepackage{amsmath,caption}
\usepackage{url,pdfpages,xcolor,framed,color}
\usepackage{a4wide}

\theoremstyle{plain}

\newtheorem{thr}{Theorem}[section]

\newtheorem{q}[thr]{Question}
\newtheorem{lem}[thr]{Lemma}
\newtheorem{prop}[thr]{Proposition}

\theoremstyle{definition}
\newtheorem{defi}[thr]{Definition}
\newtheorem{cor}[thr]{Corollary}

\def\P{\mathcal{P}}
\def\C{\mathcal{C}}

\newcommand{\DC}[2]{DC_{#1,#2}}
\newcommand{\DP}[2]{DP_{#1,#2}}

\title{An asymptotic resolution of a problem of Plesn\'{i}k}
\author{Stijn Cambie\footnote{Department of Mathematics, Radboud University Nijmegen, Postbus 9010, 6500 GL Nijmegen, The Netherlands. Email: \href{mailto:S.Cambie@math.ru.nl}{S.Cambie@math.ru.nl}. This work has been supported by a Vidi Grant of the Netherlands Organization for Scientific Research (NWO), grant number $639.032.614$.} }%
\date{}

\begin{document}
	\definecolor{xdxdff}{rgb}{0.49019607843137253,0.49019607843137253,1.}
	\definecolor{ududff}{rgb}{0.30196078431372547,0.30196078431372547,1.}
	
	\tikzstyle{every node}=[circle, draw, fill=black!50,
	inner sep=0pt, minimum width=4pt]
	
	\maketitle

	\begin{abstract}

		Fix $d \ge 3$.
		We show the existence of a constant $c>0$ such that any graph of diameter at most $d$ has average distance at most $d-c \frac{d^{3/2}}{\sqrt n}$, where $n$ is the number of vertices.
		Moreover, we exhibit graphs certifying sharpness of this bound up to the choice of $c$.
		This constitutes an asymptotic solution to a longstanding open problem of Plesn\'{i}k. Furthermore we solve the problem exactly for digraphs if the order is large compared with the diameter.

	\end{abstract}


\section{Introduction}

The average over the distances between all pairs of vertices is a fundamental parameter of a graph or network. Due to its basic character and applicability, it has arisen in diverse contexts, including efficiency of information, mass transport, molecular structure, and complex network topology, cf.~e.g.~\cite{CL}. This notion has been studied as early as 1947~\cite{Wie47}, but mathematically it is not yet fully understood. Our task in this paper is to essentially settle one of the most basic questions concerning its extremal behaviour in graphs of given diameter.

In 1984, J\`an Plesn\'{i}k~\cite{P84} determined the minimum average distance among all graphs of order $n$ and diameter $d$. He did this both for graphs and digraphs and characterized the extremal graphs. In Section~\ref{sec1} we state his results and give an alternative proof.

Determining sharp upper bounds depending on $n$ and $d$ has proven to be much more difficult. An open problem that Plesn\'{i}k had already asked was, `What is the maximum average distance among graphs of order $n$ and diameter
$d$?', both in the case of graphs and digraphs. 

After little progress, DeLaVi\~{n}a and Waller~\cite{conj} conjectured the following more concrete statement,
`Let $G$ be a graph with diameter $d>2$ and order $2d+1$ vertices. Then the average distance of $G$ is not larger than the average distance of the graph $C_{2d+1}$.'. 
This also remains open.

In 2014, Mukwembi and Vetr\'{i}k~\cite{MV} gave asymptotically sharp upper bounds for the average distance for trees with diameter $d$ up to $6$. 

In this paper, we solve the problem of Plesn\'{i}k in general for every $d$ and asymptotically as $n$ goes to infinity.
Our main results are summarized in the following two theorems.
\begin{thr}
There exist positive constants $c_1$ and $c_2$ such that for any $d \ge 3$ the following hold.
The maximum average distance among all graphs of diameter $d$ and order $n$ is between $d-c_1 \frac{d^{3/2}}{\sqrt n}$ and $d-c_2 \frac{d^{3/2}}{\sqrt n},$ i.e. it is of the form $d-\Theta\left( \frac{d^{3/2}}{\sqrt n} \right).$
\end{thr}

The proof for this result is given in Section~\ref{sec2}.
In Section~\ref{sec3}, we show slightly stronger upper bounds for trees.
This extends the results of Mukwembi and Vetrik~\cite{MV}.
Theorem~\ref{goeieConstantTrees} shows that in this case we can find constants $c_1$, $c_2$ which are fairly close for large $d$ and $n \to \infty$.

In Section~\ref{sec4}, we also settle the digraph version of the problem.

\begin{thr}\label{mainresdigraph}
	Given some integers $d \ge 2$ and $n>d$, the maximal possible total distance of a digraph with order $n$ and diameter $d$ is of the form $dn^2-d^2n+\Theta_d(1),$ and so the maximum average distance is of the form $d-\Theta\left(\frac{d^{2}}{ n}\right).$
\end{thr}

A more precise and asymptotically extremal statement is given in Theorem~\ref{digraph_d_1}.

The main first step in the proof of each of these results is to devise a graph or digraph which is almost extremal.
For this, we want many pairs of vertices which are of distance $d$ from one another.
In the graph case, we take many subtrees with many leaves. When the diameter is even, we just combine them into one tree.
When the diameter is odd, we use a central clique so that the distance between leaves of different subtrees are of distance $d$.
The construction is sketched in Figure~\ref{fig:graph}.
For some intuition about this construction, take two vertices at random.
Since the number of leaves is large, the probability that both vertices are leaves is large.
Similarly, since we have many subtrees, the probability that both leaves are in different subtrees is large.
Hence the probability that two vertices are at maximal distance is large, implying that the average distance is close to $d$ for this construction.
In the digraph case, the construction is even simpler. See Figure~\ref{fig:digraph_d}.
Every two vertices $\ell_i$ and $\ell_j$ are at distance $d$.
When $n$ is large and we choose two random vertices, the probability that they are both labeled with $\ell$ is large.
Hence the average distance will be close to $d$ again.

In the other direction, we take a graph of diameter $d$ and order $n$.
The idea is that many pairs of vertices cannot be at distance $d$ from each other.
If almost all vertices are at distance $d$ from a certain vertex $v$, their paths towards $v$ have many points in common and so the distance between these vertices is small. To make this rigorous, we apply the pigeonhole principle.

For the digraph case, we need another strategy, since we cannot use the edges in both directions to get short paths between vertices.
In this case we see that if there are many ordered pair of vertices at distance $d$, then the distance between some ordered pairs of vertices on the shortest paths are smaller than $d$.
We use this fact in a rigorous, structured  way to find a vertex $u$ such that for almost all other vertices $v$ we have $d(u,v)=d(v,u)=d$. From that, we can recover the structure of the extremal digraph.

\subsection{Notation}

A graph will be denoted by $G=(V,E)$ and 
a digraph will be denoted by $D=(V,A).$
The order $\lvert V \rvert$ will be denoted by $n$.

A cycle or directed cycle of length $k$ will be denoted by $C_k$ and $K_n$ will denote a complete graph or complete digraph on $n$ vertices.

Let $d(u, v)$ denote the distance between vertices $u$ and $v$ in a graph $G$ or digraph $D$, i.e. the number of edges in a shortest path from $u$ to $v$. 
The diameter of a graph or digraph on vertex set $V$ equals $\max_{u,v \in V} d(u,v).$
The total distance, also called the Wiener index, of a graph $G$ equals the sum of distances between all unordered pairs of vertices, i.e. $W(G)=\sum_{\{u,v\} \subset V} d(u,v).$
The average distance of the graph is $\mu(G)=\frac{W(G)}{\binom{n}{2}}$. 
The Wiener index of a digraph equals the sum of distances between all ordered pairs of vertices, i.e. $W(D)=\sum_{(u,v) \in V^2} d(u,v).$
The average distance of the digraph is $\mu(D)=\frac{W(D)}{n^2-n}.$

The distance between two subsets $X,Y \subset V$ is denoted by $d(X,Y)= \max\{d(x,y) \mid x \in X, y \in Y\}.$

For some vertex $v \in V(G)$, we denote its $k^{th}$ neighborhood with $N_k(v)=\{u \in V(G) \mid 0<d(u,v) = k\}$. We will also use $N_{\le k}(v)=\{u \in V(G) \mid 0< d(u,v) \le k\}$ and $N_{I}(v)=\{u \in V(G) \mid d(u,v) \in I\}$ for some subset or interval $I$.

In the remainder we prefer to state and prove results in terms of the Wiener index. 
Since the average distance is just a scaling of the Wiener index with a factor $\binom{n}{2}$ or $2 \binom{n}{2}$, the results can be easily interpreted in terms of the average distance $\mu$.

\begin{defi}
	Given a graph $G$ and a vertex $v$ of $G$, the blow-up of $v$ by a graph $H$ is constructed as follows.
	Take $G \backslash v$ and connect all initial neighbours of $v$ with all vertices of a copy of $H.$	
	When taking the blow-up of a vertex $v$ of a digraph $D$ by a digraph $D'$, a directed edge between a vertex $w$ of $D \backslash v$ and a vertex $z$ of $D'$ is drawn if and only if initially there was a directed edge between $w$ and $v$ in the same direction.
\end{defi}

\begin{defi}
	The sum of two graphs $G_1=(V,E_1)$ and $G_2=(V,E_2)$ is defined as $G_1+G_2=(V,E_1 \cup E_2).$
	The sum of two digraphs on the same vertex set is defined similarly.
\end{defi}

The statement $f(x)=O(g(x))$ as $x \to \infty$ implies that there exists fixed constants $x_0, M>0$, such that for all $x \ge x_0$ we have $\lvert f(x) \rvert \le M \lvert g(x) \rvert .$
Analogously, $f(x)=\Omega(g(x))$ as $x \to \infty$ implies that there exists fixed constants $x_0, M>0$, such that for all $x \ge x_0$ we have $\lvert f(x) \rvert \ge M \lvert g(x) \rvert.$
If $f(x)=\Omega(g(x))$ and $f(x)=O(g(x))$ as $x \to \infty$, then one uses $f(x)=\Theta(g(x))$ as $x \to \infty$.
We also write $f(x)=g(x)+O(h(x))$ when we know that $f(x)-g(x)=O(h(x))$ and analogously for $\Theta$.
Similarly, we have $O_d, \Omega_d$ and $\Theta_d$ if $x_0$ and $M$ may vary with $d$. For example $f(x)=O_d(g(x))$ as $x \to \infty$ implies that for every $d$, there exist constants $x_0, M>0$, such that for all $x \ge x_0$ we have $\lvert f(x) \rvert \le M \lvert g(x) \rvert.$

\section{Plesn\'{i}k's lower bounds}\label{sec1}

The results of Plesn\'{i}k~\cite{P84} are stated in the following theorems. 
We also give a short alternative proof for the result in the graph case, which can analogously be used to prove the digraph case as well.

\begin{thr}[\cite{P84}, Theorem 2]\label{p1}
	Let $G$ be a graph with $n$ vertices and diameter $d$. 
	Then 
	
	$$
	W(G) \ge 
	\begin{cases}
	\binom{d+2}{3}+\frac{1}4 (n-d-1)(2n+d^2+1) \hfill \mbox{ if d is odd}\\
	\binom{d+2}{3}+\frac{1}4 (n-d-1)(2n+d^2) \hfill \mbox{ if d is even}
	\end{cases}
	$$
	Moreover the extremal graphs are exactly the maximal graphs of diameter $d$ with all noncentral layers being trivial, i.e. any extremal graphs can be created by starting from a path of length $d$ and taking blow-ups with cliques at its central vertices.
	
\end{thr}

\begin{proof}
	When $d=1$, the unique extremal graph is obviously $K_n$.
	Now consider $d>1$ and an extremal graph $G$. As $G$ has diameter $d$, it has two vertices $u_0$ and $u_d$ with $d(u_0,u_d)=d$. Take a shortest path $\P=u_0 u_1 \ldots u_{d-1}u_d$ between them.
	Note that every vertex $v$ which does not belong to $\P$ and every $0\le i < \frac d2$ satisfies $d(u_i,v)+d(v,u_{d-i}) \ge d(u_i,u_{d-i})=d-2i.$
	In particular when $d$ is even, we have
	$$\sum_{j=0}^{d} d(v, u_j) \ge 1+ \sum_{i=0}^{\frac d2 -1} d-2i =1+\frac{d^2+2d}{4}$$ and similarly for $d$ being odd, we have 
	$$\sum_{j=0}^{d} d(v, u_j) \ge 2+ \sum_{i=0}^{\frac{d-1}2 -1} d-2i =1+\frac{(d+1)^2}{4}.$$
	Together with $\sum_{v,w \in V(G \backslash \P)} d(v,w) \ge \binom{n-d-1}{2}=\frac{(n-d-1)(2n-2d-4)}{4}$ and $W(\P)=\binom{d+2}{3}$ we get the bounds on $W(G)$.
	Equality occurs if and only if equality occurs in every step, from which the characterization of the extremal graphs follows as well.
\end{proof}

\begin{thr}[\cite{P84}, Theorem 3]\label{p2}
	Let $D$ be a digraph with $n$ vertices and diameter $d$. 
	Then 
	
	$$
	W(D) \ge 
	\begin{cases}
	\dfrac{d(d+1)(d+5)}6+\frac{1}4 (n-d-1)(4n+(d+1)^2 ) \hfill \mbox{ if d is odd}\\
	\dfrac{d(d+1)(d+5)}6+\frac{1}4 (n-d-1)(4n+d^2+2d) \hfill \mbox{ if d is even}
	\end{cases}
	$$
	Moreover the extremal digraphs are exactly the maximal digraphs of diameter $d$ with all noncentral layers being trivial, i.e. any extremal digraph can be created by some blow-ups by cliques at the central vertices of a digraph $D'$. Here $D'$ is the sum of a transitive tournament on $d+1$ vertices and the unique longest path in its complement.
\end{thr}

Note that using the same ideas, we easily get an alternative proof for a result of Ore~\cite{Ore68} (Theorem $3.1$ there) and its digraph version as well, which we state for completeness.

\begin{thr}
	Let $D$ be a digraph of order $n$ and diameter $d \ge 2$. 
	Then its size $$\lvert A(D) \rvert \le (n-d-1)(n+2)+\binom{d+2}{2}-1.$$
	Equality holds if and only if $D$ can be created by blow-ups at $1$ or $2$ consecutive non-end vertices of a digraph $D'$, which is the sum of a transitive tournament on $d+1$ vertices and the unique longest path in its complement.
\end{thr}

\section{An asymptoticaly sharp upper bound}\label{sec2}

In this section we solve the problem of Plesn\'{i}k~\cite{P84} for graphs asymptotically, in the sense that we determine the order of magnitude of the gap with respect to the trivial upper bound. 

\begin{thr}\label{mainres}
The maximum Wiener index of a graph $G$ of order $n$ and diameter $2$ equals $(n-1)^2$. Equality holds if only if the graph is a star.

The maximum Wiener index of a graph with order $n$ and diameter $d \ge 3$ is $\frac{d}{2}n^2-d^{3/2}\Theta( n^{3/2}).$
\end{thr}

\begin{proof}
For the first part, note that if any vertex of $G$ has degree $1$, the condition on the diameter implies that the graph contains a star and so the maximum is attained by the star.
Otherwise, every vertex has at least two neighbors at distance $1$. So there are at least $n$ pairs of vertices at distance $1$ of each other, implying that $W(G) \le 2\binom{n}{2}-n=n^2-2n<(n-1)^2$.

Alternatively, for a graph with diameter $2$, one has that $W(G)=2\binom{n}{2}-\lvert E \rvert \le (n-1)^2$. Equality does hold if and only if $G$ is a tree with diameter $2$, i.e. $G$ is a star.

%
%

The second part is a consequence of Theorem~\ref{thr1} and Theorem~\ref{thr2}.
\end{proof}

\begin{thr}\label{thr1}
	
There is a function $f_1$ for which $f_1(d)=\frac{1}{12\sqrt 3}d^{3/2}-O(\sqrt d)$ as $d \to \infty$ for which the following hold.
For $d \ge 3$ fixed, the Wiener index of a graph with order $n$ and diameter $d$ is at most, as $n \to \infty$,
	$$\frac d2 n^2 -f_1(d) n^{3/2} + O_d(n).$$
%
%
%
\end{thr}

\begin{proof}

Let $G=(V,E)$ be a graph of order $n$ and diameter $d$.
Look to the set \\$S=\left\{u \in V \colon \left\lvert N_{\le \frac 56 d}(u) \right\rvert \ge   \sqrt{\frac{1}{3}} \sqrt{dn} \right\}.$
If $S=V$, then there are at least $\frac{1}{2}n \sqrt{\frac{1}{3}} \sqrt{dn}$ pairs with $d(u,v) \le \frac 56 d$ and so 
$$W(G) \le d \binom{n}{2}-\frac{1}{2}n \sqrt{\frac{1}{3}}\sqrt{dn} \cdot \frac 16 d =\frac{d}{2}(n^2-n)-\frac{1}{12 \sqrt 3}d^{3/2}n^{3/2}.$$

In the other case we take $u \in V \backslash S$.
Then $b= \left \lvert N_{]\frac 56 d,d](u)} \right \rvert \ge n-\sqrt{\frac{1}{3}}\sqrt {dn}$, while $\left \lvert N_{[\frac 23 d,\frac 56 d]}(u) \right \rvert \le \sqrt{\frac{1}{3}} \sqrt{dn} $.
By the pigeonhole principle, there will be some $k \in [\frac 23 d,\frac 56 d]$ such that $a=\lvert N_k(u) \rvert \le \frac{ \sqrt{\frac{1}{3}} \sqrt{dn}}{ \frac16 d - \frac 13},$ since there are at least $\frac16 d - \frac 13$ integers in $[\frac 23 d,\frac 56 d]$.
Every element in $N_{[\frac 56d,d]}(u)$ is connected by a path of length at most $\frac 13 d$ to some vertex of $N_k(u)$.
Assign every element $v$ in $N_{[ \frac 56 d,d]}(u)$ to exactly one such element $w$ (with $d(v,w) \le \frac 13 d$) in $N_k(u)$.
For every vertex $w_i \in N_k(u)$ with $1 \le i \le a$, let $x_i$ be the number of elements in $N_{]\frac 56 d,d]}(u)$ that have been assigned to $w_i.$
By the triangle inequality, the distance between any two elements in $N_{]\frac 56 d,d]}(u)$ that are assigned to the same $w_i$ is at most $\frac 23  d$, instead of $d$ and so $d \binom{n}2$ overestimates the Wiener index by at least $\frac 13 d$ for this pair.
The number of such pairs is at least
\begin{align*}
\sum_{i=1}^a \binom{x_i}{2} &= \frac12 \left( \sum_{i=1}^a x_i^2 - \sum_{i=1}^a x_i \right)\\
&\ge \frac12 \left( a \left(\frac{b}{a}\right)^2-b \right)\\
&\ge \frac12 \left( n-\sqrt{\frac{1}{3}}\sqrt {dn} \right) \left(\dfrac{ n-\sqrt{\frac{1}{3}}\sqrt {dn}}{ \frac{ \sqrt{\frac{1}{3}} \sqrt {dn}}{ \frac16 d - \frac 13}  } -1 \right) \\
&= \left( \frac{1}{4\sqrt 3}\sqrt d-O\left(\frac1{\sqrt d}\right) \right) n^{3/2} - O_d(n)
\end{align*}
for large enough $n$. Here we used the inequality between the quadratic mean and the arithmetic mean (QM-AM).
Hence the Wiener index has been overestimated by at least \\$\left( \frac{1}{12\sqrt 3}d^{3/2}-O(\sqrt d) \right) n^{3/2} - O_d(n)$, from which the conclusion follows.
\end{proof}

\begin{thr}\label{thr2}
	There is a function $f_2$ for which $f_2(d)=\frac{3 \sqrt{2}}{8}d^{3/2} +O(\sqrt{d})$ as $d \to \infty$ for which the following hold.
	For $d \ge 3$ fixed, there exist a graph $G$ with order $n$ and diameter $d$ with Wiener index at least $\frac{d}{2}n^2 -f_2(d) n^{3/2} - O_d(n)$ as $n \to \infty.$
%
\end{thr}

\begin{proof}
Take $n$ and $\ell$, with $d=2\ell$ or $d=2\ell+1$.
We take $k = \lfloor \sqrt{ \frac n{\ell} } \rfloor.$ 
We construct a graph as in Figure~\ref{fig:graph}.
When $d$ is even, $C$ is a single vertex connected to $k$ branches.
When $d$ is odd, we take $C$ equal to the graph $K_k$ and every vertex of that $K_k$ being connected to exactly one of $k$ branches.
Every branch is a broom (a concatenation of a path and a star), with the number of leaves equal to $a_i$.
We take every $a_i$ being at least $\lfloor \frac{n-k(\ell-1)-\lvert C \rvert }{k} \rfloor \ge \frac{n}{k}-\ell-1$, with the condition that $\sum_{i=1}^k a_i +k(\ell-1)+\lvert C \rvert=n.$
Then \belowdisplayskip=-12pt
\begin{align*}
W(G) &\ge \sum_{i \not= j} a_i a_j d + (k-1)\frac32 \ell(\ell-1)\left( \sum_{i=1}^k a_i \right)\\
&\ge d \binom{k}{2} \left(\frac{n}{k}-\ell-1 \right)^2 + \frac 32 (k-1)(\ell-1)\ell \left(n-k(\ell-1)-\lvert C \rvert \right)\\
&= \frac d2 n^2 \left(1-\frac 1k\right) -dk\ell n+\frac32 k\ell^2n - O_d(n)-O(\sqrt d)n^{3/2}\\
&=\frac d2 n^2 -\frac{\sqrt 2}{4}d^{3/2}n^{3/2} -\frac{1}{\sqrt 2}d^{3/2}n^{3/2}+\frac 32 \frac{1}{2 \sqrt 2}d^{3/2}n^{3/2}- O_d(n)-O(\sqrt d)n^{3/2}\\
&=\frac{d}{2}n^2 - \left( \frac{3 \sqrt{2}}{8}d^{3/2} +O(\sqrt{d})\right) n^{3/2} - O_d(n).
\end{align*}  \end{proof}

\begin{figure}[h]
\centering
\begin{tikzpicture}\centering [thick,scale=0.4]%
\draw \foreach \x in {0,45,90,180} {
	(\x:0)--  (\x:0.5)  node {}
	(\x:0.5)--  (\x:1)  node {}
};
\draw (0:0) node{};
\draw[dotted] \foreach \x in {0,45,90,180} {
	(\x:0.55)--(\x:3.75)  
};

\draw \foreach \a in {5,-5} \foreach \x in {0,45,90,180} {
	(\x:3.75)--(\x+\a:4.25) node {}
};
\draw \foreach \x in {0,45,90,180} {
	(\x:3.75)  node {}
	
};

\draw[dotted]  \foreach \x in {0,45,90,180} {
	(\x-5:4.25)--(\x+5:4.25) 
};

\draw [fill=ududff] (0,0) circle (2.5pt);

\coordinate [label=center:$C$] (A) at (0,-0.35); 

\coordinate [label=left:$l-1$] (A) at (3,0.3); 
\coordinate [label=center:$l-1$] (A) at (0.5,2.25); 
\coordinate [label=center:$l-1$] (A) at (-2.25,0.25); 
\coordinate [label=right:$l-1$] (A) at (40:2.25); 

\coordinate [label=center:$a_1$] (A) at (0:4.5); 
\coordinate [label=center:$a_2$] (A) at (45:4.5); 
\coordinate [label=center:$a_3$] (A) at (90:4.5); 
\coordinate [label=center:$a_k$] (A) at (180:4.5);

\draw[dotted]  (0.75,0) arc (0:180:0.75) ;
\coordinate [label=center:$k$] (A) at (135:0.95); 

\end{tikzpicture}

\caption{Graph obtaining upper bound}
\label{fig:graph}
\end{figure}
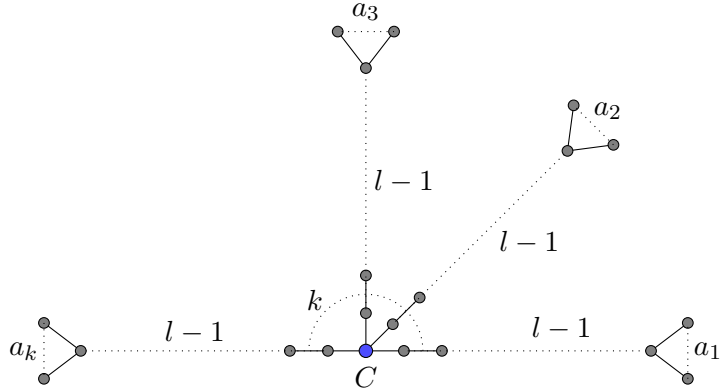

We determine the gap to the trivial upper bound $\frac d2 n^2$ to within a factor $5$ for large $d$ and $n$. The improvement on the lower bound is done by looking to the size of neighbourhoods $N_m(u)$ in a more continuous way than was done in Theorem~\ref{thr1}. The improved upper bound follows from a modification of the construction in Figure~\ref{fig:graph}.
For trees, we can further sharpen this, see Theorem~\ref{goeieConstantTrees} below.

\begin{thr}\label{lelijk}
Writing $W_{n,d}$ for the maximum Wiener index of a graph with order $n$ and diameter $d$, 
$$\liminf_{d \to \infty} \liminf_{n \to \infty} \dfrac{\frac d2 n^2 - W_{n,d}}{d^{3/2} n^{3/2} } \ge \frac{1}{2 \sqrt{24}}>0.1 \mbox{ and}$$
$$\limsup_{d \to \infty} \limsup_{n \to \infty} \dfrac{\frac d2 n^2 - W_{n,d}}{d^{3/2} n^{3/2} } < \frac{29}{60}$$ 
\end{thr}

\begin{proof}

Fix any constant $0<c< \frac{1}{2 \sqrt{24}}.$

Take any $d$ and $\epsilon=\frac 1d$ such that $ c<  (1- \epsilon) \frac{1}{2 \sqrt{24}} \left(1-\frac {6}{d}-\frac {6}{d^2} \right)$.
For every $n>d$, we will prove that $W_{n,d} < \frac d2 n^2 -c d^{3/2}n^{3/2} + O_d(n)$ from which the result follows.

Take an optimal graph $G$ with Wiener index $W_{n,d}$.

If $\lvert N_{d}(u) \rvert <n- 2\frac{c}{\epsilon} d^{3/2}\sqrt n $ for at least $\epsilon n$ vertices  $u \in V$, the result follows.

Let $U$ be the set of vertices $u$ satisfying $\lvert N_{d}(u) \rvert \ge n- 2\frac{c}{\epsilon} d^{3/2}\sqrt n$ and assume $\lvert U \rvert > (1-\epsilon) n.$

If $\lvert N_{m}(u) \rvert > \frac{(2m-d)}{d}\sqrt{ \frac{24 n}{d} }$ for all $u \in U$ and for every $\frac d2 \le m<d$, then the Wiener index is at most

\begin{align*}
d \binom{n}{2}- \frac{1}{2} (1-\epsilon) n \sum_{ d/2 \le m \le d-1} (d-m)\frac{(2m-d)}{d}\sqrt{ \frac{24 n}{d} }.
\end{align*}

By some approximation theory for right Riemann sums, we know
\begin{align*}
\left \lvert \int_{d/2}^d (d-x)(2x-d) dx - \sum_{ d/2 \le m \le d-1} (d-m)(2m-d)  \right \rvert
\le \frac d4  + \frac{d^2}4
\end{align*}
and thus 
\begin{align*}
W(G) &\le d \binom{n}{2}-\sqrt{6}(1-\epsilon)\frac{n^{3/2}}{d^{3/2}}\left( \frac{d^3}{24}-\frac{d^2}{4}-\frac{d}{4} \right)
\\&\le d \binom{n}{2} -cd^{3/2}n^{3/2}
\end{align*}
from which the result follows in this case.

In the remaining case, there is some $d/2 \le m \le d-1 $ and $u \in U$ such that $a=\lvert N_{m}(u) \rvert \le \frac{(2m-d)}{d}\sqrt{ \frac{24 n}{d} }$, while $b=\lvert N_d(u) \rvert \ge n- 2\frac{c}{\epsilon} d^{3/2}\sqrt n.$
Assign every vertex $v$ in $N_d(u)$ to one vertex $w \in N_{m}(u)$ for which $d(v,w)=d-m.$
Similarly as before in the proof of Theorem~\ref{thr1}, there are at least $$\frac 12 b \left ( \frac b a - 1 \right) = \frac{d^{3/2} n^{3/2}}{2 \sqrt{24}(2m-d)} +O_d(n)$$
pairs of vertices in $N_d(u)$ that are assigned to the same vertex in $N_{m}(u)$.
For those pairs the Wiener index has been overestimated by at least $d-2(d-m)=2m-d$ and thus $$W(G) \le d \binom{n}{2} - \frac{d^{3/2} n^{3/2}}{2 \sqrt{24}} +O_d(n).$$

For the second inequality, we modify the branches in the construction of Figure~\ref{fig:graph} to get an improvement of Theorem~\ref{thr2}.
We take $k=\lfloor \sqrt {\frac n{\ell} } \rfloor=\sqrt {\frac n{\ell} }-c_1$ for some $0\le c_1<1$ again into branches of the form as shown in Figure~\ref{fig:subtrees}, where the height of the rooted tree with respect to $v$ is $\ell$. Every branch splits in two at every height $\ell -\lfloor c^s \ell \rfloor $ for $1 \le s \le  t = \lfloor -\log_c(\ell) \rfloor  $ and $c < 0.5$ some fixed constant we determine later.
At height $\ell-1$, every branch splits again, this in such a way that the degrees of two vertices at height $\ell-1$ differ by at most $1.$
The number of vertices up to height $\ell-1$ from $v$, will be equal to
$$\ell(1+c+2c^2+4c^3+\ldots+2^{t-1}c^{t})+O(2^t)= \frac {\ell}2 \left( 1+ \frac1{1-2c} \right)+o(\ell).$$
At the end of the branch we have $A=\frac nk - \frac {\ell}2 \left( 1+ \frac1{1-2c} \right)+o(\ell)= \sqrt{n \ell }+\left( c_1- \frac {1}2 \left( 1+ \frac1{1-2c} \right) \right)\ell+  o(\ell)$ leaves.

The sum of distances between leaves of different branches equals
\begin{align*}
d \binom{k}{2} A^2.
\end{align*}

The sum of distances from one fixed leaf to all non-leafs from another fixed branch, is up to some $o(\ell^2)$-term equal to
\begin{align*}
2\ell \frac {\ell}2 \left(1+ \frac{1}{1-2c}\right) -  \ell\frac {\ell}2 - \sum_{m=1}^t 2^{m-1} c^m\ell \frac {c^m\ell}{2} \\
= \left(\frac 34 + \frac{1}{1-2c} -\frac14 \frac1{1-2c^2} \right) \ell^2 .
\end{align*}

The total sum of distances between leaves of one branch and non-leaf elements in other branches, is up to some $o\left(d^{3/2}n^{3/2}\right)$-terms equal to
$$k(k-1) A \left(\frac 34 + \frac{1}{1-2c} -\frac14 \frac1{1-2c^2} \right) \ell^2.$$
Furthermore, the sum of distances between leaves of the same branch, is 
\begin{align*}
\sum_{m=1}^{t} \left(2c^m \ell+O(1) \right) 2^{m-1}  \left(\frac A{2^m}+O(1)\right)^2 \\
=\frac c2 \frac{1}{1-\frac c2} \ell A^2 + o(\ell A^2)+O_{\ell}(A)
\end{align*}
and there are $k$ such branches.

Summing over the previous three expressions, we see that the Wiener index is at least 
\begin{align*}
\frac d2 \left (\frac n{\ell} -(2c_1+1)\sqrt{ \frac n{\ell}} \right )\left( \ell n +\left(2c_1-1-\frac{1}{1-2c}\right) \ell^{3/2} \sqrt n \right) \\+ \frac n{\ell} \sqrt{n\ell}\left(\frac 34 + \frac{1}{1-2c} -\frac14 \frac1{1-2c^2} \right) \ell^2
+\sqrt {\frac n{\ell}} \frac c2 \frac{1}{1-\frac c2} \ell n\ell+o(d^{3/2}n^{3/2})\\
= \frac d2 n^2 + \left( \frac{-5}{4}-\frac14 \frac1{1-2c^2}+\frac{c}{2-c}   \right) n^{3/2}\frac{d^{3/2}}{2^{3/2}}+o(d^{3/2}n^{3/2}) 
\end{align*}

Choosing $c=0.3825$ does the job.
Note that we can find the optimal value for $c$ as the solution of a fifth degree polynomial, by taking $f'(c)=0$ for $f(c)=-\frac14 \frac1{1-2c^2}+\frac{c}{2-c} .$
\end{proof}

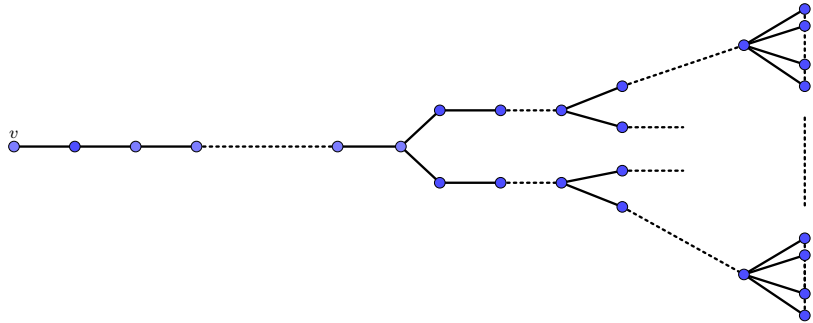
\begin{figure}[h]
	\centering

	\scalebox{0.8}{
		\begin{tikzpicture}[line cap=round,line join=round,>=triangle 45,x=1.0cm,y=1.0cm]
		\clip(1.8,-2.95) rectangle (16,2.4);
		\draw [line width=1.1pt,dotted] (5.,0.)-- (7.32,0.);
		\draw [line width=1.1pt] (2.,0.)-- (3.,0.);
		\draw [line width=1.1pt] (3.,0.)-- (4.,0.);
		\draw [line width=1.1pt] (4.,0.)-- (5.,0.);
		\draw [line width=1.1pt] (7.32,0.)-- (8.36,0.);
		\draw [line width=1.1pt,dotted] (15.0,2.28)-- (15.0,1.36);
		\draw [line width=1.1pt, dotted] (15.,1.)-- (15.0,1.36);
		\draw [line width=1.1pt] (14.0,1.68)-- (15.0,2.28);
		\draw [line width=1.1pt] (14.0,1.68)-- (15.,2.);
		\draw [line width=1.1pt] (14.0,1.68)-- (15.0,1.36);
		\draw [line width=1.1pt] (14.0,1.68)-- (15.,1.);
		\draw [line width=1.1pt] (9.,0.6)-- (10.0,0.6);
		\draw [line width=1.1pt] (9.0,-0.6)-- (10.0,-0.6);
		\draw [line width=1.1pt] (11.0,0.6)-- (12.,1.);
		\draw [line width=1.1pt] (11.0,0.6)-- (12.,0.32);
		\draw [line width=1.1pt] (11.0,-0.6)-- (12.0,-0.4);
		\draw [line width=1.1pt] (11.0,-0.6)-- (12.,-1.);
		\draw [line width=1.1pt,dotted] (12.,1.)-- (14.0,1.68);
		\draw [line width=1.1pt,dotted] (10.,0.6)-- (11.0,0.6);
		\draw [line width=1.1pt] (9.,0.6)-- (8.36,0.);
		\draw [line width=1.1pt] (8.36,0.)-- (9.,-0.6);
		\draw [line width=1.1pt,dotted] (10.0,-0.6)-- (11.0,-0.6);
		\draw [line width=1.1pt,dotted] (12.,0.32)-- (13.,0.32);
		\draw [line width=1.1pt,dotted] (12.0,-0.4)-- (13.0,-0.4);
		\draw [line width=1.1pt,dotted] (15.0,-1.52)-- (15.,-2.44);
		\draw [line width=1.1pt,dotted ] (15.0,-2.8)-- (15.0,-2.44);
		\draw [line width=1.1pt] (14.0,-2.12)-- (15.0,-1.52);
		\draw [line width=1.1pt] (14.0,-2.12)-- (15.0,-1.8);
		\draw [line width=1.1pt] (14.0,-2.12)-- (15.0,-2.44);
		\draw [line width=1.1pt] (14.0,-2.12)-- (15.0,-2.8);
		\draw [line width=1.1pt,dotted] (15.0,-2.44)-- (15.0,-1.8);
		\draw [line width=1.1pt,dotted] (12.,-1.)-- (14.0,-2.12);
		\draw [line width=1.1pt,dotted] (15.0,0.48)-- (15.,-1.);
		\begin{scriptsize}
		\draw [fill=xdxdff] (2.,0.) circle (2.5pt);
		\draw [fill=ududff] (3.,0.) circle (2.5pt);
		\draw [fill=xdxdff] (4.,0.) circle (2.5pt);
		\draw [fill=xdxdff] (5.,0.) circle (2.5pt);
		\draw [fill=xdxdff] (7.32,0.) circle (2.5pt);
		\draw [fill=xdxdff] (8.36,0.) circle (2.5pt);
		\draw [fill=ududff] (9.,0.6) circle (2.5pt);
		\draw [fill=ududff] (9.0,-0.6) circle (2.5pt);
		\draw [fill=ududff] (10.0,-0.6) circle (2.5pt);
		\draw [fill=ududff] (11.0,-0.6) circle (2.5pt);
		\draw [fill=ududff] (11.0,0.6) circle (2.5pt);
		\draw [fill=ududff] (10.0,0.6) circle (2.5pt);
		\draw [fill=ududff] (12.,1.) circle (2.5pt);
		\draw [fill=ududff] (12.,0.32) circle (2.5pt);
		\draw [fill=ududff] (12.0,-0.4) circle (2.5pt);
		\draw [fill=ududff] (12.,-1.) circle (2.5pt);
		\draw [fill=ududff] (14.0,1.68) circle (2.5pt);
		\draw [fill=ududff] (15.0,2.28) circle (2.5pt);
		\draw [fill=ududff] (15.,2.) circle (2.5pt);
		\draw [fill=ududff] (15.0,1.36) circle (2.5pt);
		\draw [fill=ududff] (15.,1.) circle (2.5pt);
		\draw [fill=ududff] (14.0,-2.12) circle (2.5pt);
		\draw [fill=ududff] (15.0,-1.52) circle (2.5pt);
		\draw [fill=ududff] (15.0,-1.8) circle (2.5pt);
		\draw [fill=ududff] (15.,-2.44) circle (2.5pt);
		\draw [fill=ududff] (15,-2.8) circle (2.5pt);
		\coordinate [label=center:$v$] (A) at (2,0.2); 
		\end{scriptsize}
		\end{tikzpicture}}
	\caption{Good approximation form optimal subtrees}
	\label{fig:subtrees}
\end{figure}

\section{A sharper bound for trees}\label{sec3}

\begin{lem}
A tree with maximum Wiener index among all trees of order $n$ and diameter $d$ will not contain any leaf that has distance to a central vertex strictly smaller than $\frac{d-1}{2}$.
\end{lem}

\begin{proof}
Assume by the contrary that some extremal tree $T$ does. Let $c$ be a central vertex and $\ell$ a leaf for which $d(c, \ell) < \frac{d-1}{2}.$ Let $v$ be the vertex closest from $\ell$ that has degree at least $3$. This vertex does exist since the extremal tree is not a path. Let $x=d(\ell,v)$ be the length of the path from $v$ to $\ell.$ Let $b$ be the neighbour of $v$ which is on this path as well (possible $b=\ell$).
Note that $T \backslash v$ has at least three components by definition.
Let $C$ be the smallest component not containing $\ell.$
Let $w$ be the neighbour of $v$ which is a vertex of $C$.

We now construct a new tree $T'$ by deleting the edge between $b$ and $v$ and adding an edge between $b$ and $w$.
Since $d_{T'}(\ell, c) \le d_{T}(\ell, c)+1 \le \frac d2$ we can note that $T'$ has also diameter $d$.
Furthermore $T'$ has a larger Wiener index than $T$ since the distance between the vertices of the path from $\ell$ to $b$ and vertices not belonging to this path nor the component $C$ has increased with $1.$
This implies $T$ was not extremal, contradiction with the original assumption.
\end{proof}

We first define some help functions which will be used in the computations.

\begin{defi}
The function $g_1$ satisfies $g_1(n)=n-1$ for all $ n \in \mathbb N$.
\\For every $k>1$ and $n \in \mathbb N$, $g_k(n)=n+g_{k-1}(n)-2 \sqrt{g_{k-1}(n)}.$
\\For every $k \in \mathbb N$, the function $h_k$ is defined by $h_k(n) =k \cdot n - \frac 43 \sqrt k (k-1)\sqrt n + \frac23(k-1)^2$ for every $n \in \mathbb N.$
\end{defi}

\begin{defi}
Let the Wienerindex of a graph $G$ whose vertices $v$ have been assigned a multiplicity $\omega(v)$ be equal to $W(G)=\sum_{\{u,v\} \subset V} \omega(u)\omega(v) d(u,v).$
The function $dW$ gives the maximal Wiener index $dW(t,n,i)$ of a rooted tree of height at most $i$ with $t$ non-root vertices having multiplicity $1$ and a single root vertex having multiplicity $n-t$.
\end{defi}

\begin{prop}\label{gh}
For every $k,n$ with $n \ge 5$, we have $g_k(n) < h_k(n).$
\end{prop}

\begin{proof}
We prove this by induction.
For $k=1$, we get $n-1 < n$ which is true.
Note that $g_k(n) \ge 4$ for every $k$ and $n\ge 5$ and thus 
\begin{align*}
g_{k+1}(n)&=n+g_{k}(n)-2 \sqrt {g_k(n)} 
\\&< n+h_k(n)-2 \sqrt{h_k(n)}
\\&\le n+ k \cdot n - \frac 43 \sqrt k (k-1)\sqrt n + \frac23(k-1)^2 -2\left(\sqrt k \sqrt{n} -\frac23(k-1)\right)
\\&\le (k+1)n-\frac43 \sqrt{k+1}k\sqrt n +\frac23 k^2
\\&=h_{k+1}(n),
\end{align*}
since $h_k(n) \ge \left(\sqrt k \sqrt{n} -\frac23(k-1)\right)^2 $, $ \sqrt{k(k+1)}  <(k+0.5) $ and $\frac{2}3 \left( k^2-(k-1)^2 \right)=\frac{4}{3} k-\frac 23 > \frac 43 (k-1),$ which proves the induction step.\end{proof}
 
\begin{prop}\label{prop5.5}
For every $t,i \ge 1$ and $n \ge 5$, the function $dW$ satisfies 
\begin{equation} \label{dW_g}
dW(t,n,i) \le t g_i(n).
\end{equation} 

\end{prop}

\begin{proof}
We prove this by induction.
When $i=1$, we have by definition that $dW(t,n,1)= 2\binom{t}{2}+t(n-t)=nt-t=t g_1(n)$.

When $i>1$, we let $k$ be the degree of the root in the optimal configuration.
So the root is connected to $k$ rooted subtrees of height at most $i-1$.
Assume the $j^{th}$ rooted subtree has $t_j$ non-root vertices.	
So $t=k+\sum_{j=1}^k t_j.$
Note that $dW(t_j,n,i-1)$ gives the difference between the Wiener index of the whole rooted tree and the rooted tree, where the $j^{th}$ subtree has been replaced by a single vertex of multiplicity $t_j+1$.
Using this, together with the induction hypothesis for $dW(t,n,i-1)$ and the inequalities between the quadratic mean, the arithmetic mean and the geometric mean (QM-AM, AM-GM), we get
\belowdisplayskip=-12pt
\begin{align*}
dW(t,n,i)&= \sum_{j=1}^k dW(t_j,n,i-1) + t(n-t)+2\sum_{1 \le j_1 < j_2 \le k} (t_{j_1}+1)(t_{j_2}+1)\\
& \le (t-k)g_{i-1}(n)+t(n-t)+t^2-\sum_{j=1}^k (t_j+1)^2\\
& \le (t-k)g_{i-1}(n)+t(n-t)+t^2-k\left(\frac tk \right)^2\\
&=   t g_{i-1}(n) +nt - \left( k g_{i-1}(n) +\frac{t^2}k \right)\\
&\le t g_{i-1}(n) +nt - 2 \sqrt{  k g_{i-1}(n) \frac{t^2}k }\\
&= t g_i(n).
\end{align*} 
\end{proof}

\begin{thr}\label{treedw}
Let $T$ be a tree of diameter $d$ with $n$ vertices.\\
If $d=2\ell$ for some integer $\ell$, we have 
$W(T) \le (n-1) g_{\ell}(n).$\\
If $d=2\ell+1$ for some integer $l$,
we have $W(T) \le (n-2)g_{\ell}(n) + \frac{n^2}{4}.$
\end{thr}

\begin{proof}
If $d$ is even, the central vertex can be seen as the root of the tree.	
Hence by definition, $W(T) = dW(n-1,n,\ell)\le (n-1) g_{\ell}(n)$.

If $d$ is odd, there are two central vertices $u$ and $v$.
Let $S_u$ be the set of vertices closer to $u$ than to $v$ (and different from $u$) and define $S_v$ similarly.
Let $S_u$ and $S_v$ have sizes $t_u$ and $t_v$ respectively.
Then the Wiener index of the tree $T$ satisfies
\belowdisplayskip=-12pt \begin{align*}
W(T) &\le dW(t_u,n,\ell)+dW(t_v,n,\ell)+(t_u+1)(t_v+1)\\
&\le t_u g_{\ell}(n) + t_v g_{\ell}(n) + \left( \frac{n}{2} \right)^2\\
&=(n-2)g_{\ell}(n) + \frac{n^2}{4}.
\end{align*} \end{proof}

As a corollary of Theorem~\ref{treedw}, we get the following.

\begin{cor}

Writing $W^T_{n,d}$ for the maximum Wiener index among all trees with order $n$ and diameter $d$, 

\begin{itemize}
	\item $W^T_{n,3}= \lfloor \frac{5n^2}{4}-3n+2 \rfloor$
	\item $W^T_{n,4} \le 2n^2-2n\sqrt{n-1}-3n+2\sqrt{n-1}+1$
	\item $W^T_{n,5} \le \frac{9}{4}n^2 -2n \sqrt{n-1} +O(n)$
	\item $W^T_{n,6} \le 3n^2-2(1+\sqrt 2) n^{3/2} +O(n)$
\end{itemize}
\end{cor}

For $d=3$, this is Theorem $3.1$ of \cite{MV} (with a correction of $1$ since $d(u,v)$ was counted twice in that document) where the floor function ensures that equality holds for all $n$.

For $d=4$, equality holds if and only if $n-1$ is a perfect square. For general $n$, the exact value of $W^T_{n,4}$ can differ with respect to the upper bound by more than $1$.

For $d=5$, the upper bound $(n-2)g_2(n)+\frac{n^2}{4}$ can not be attained as we will explain. Equality would imply that we have $t_u=t_v=\frac n2 -1$ and $dW(t_u,n,2)=t_u g_2(n)$ in the proof of Theorem~\ref{treedw}.
To have $dW(t_u,n,2)=t_u g_2(n)$, we need equality when applying the AM-GM inequality in the proof of Proposition~\ref{prop5.5}
and thus $k g_1(n)= \frac{t^2}k$ for $t=\frac n2 -1$. Nevertheless, the equality $k^2(n-1)=\left( \frac{n}{2}-1 \right)^2 $ does not hold for any integer $k$, since $n-1 \nmid (n-2)^2$ for $n \ge 3.$

For $d=6$, equality in every inequality between the arithmetic mean and the geometric mean (AM-GM) is again impossible as was in the case for $d=5.$
In \cite{MV}, the authors make considerable efforts to bound the second order term, achieving $\sqrt 6$ for this specific choice of diameter.

We have managed to obtain good estimates on the second-order term for $d$ in general as $n \to \infty$, with an error of less than $2.5$\%.

\begin{thr}\label{goeieConstantTrees}

When $d$ is even, we have as $n \to \infty$
$$ \frac d2 n^2  -\left( \frac {29}{60}d^{3/2} +o(d^{3/2})\right) n^{3/2}+ O_d(n) \le  W^T_{n,d} \le \frac d2 n^2 - \left( \frac{ \sqrt{2}}3 d^{3/2} + O(\sqrt d) \right) n^{3/2} + O_d(n),$$
while, if $d$ is odd, we have as $n \to \infty$
$$ \frac{d-0.5}2 n^2  -\left( \frac {29}{60}d^{3/2} +o(d^{3/2})\right) n^{3/2}+ O_d(n) \le  W^T_{n,d} \le \frac {d-0.5}2 n^2 - \left( \frac{ \sqrt{2}}3 d^{3/2} + O(\sqrt d) \right) n^{3/2} + O_d(n).$$
\end{thr}

\begin{proof}
The lower bounds on $W^T_{n,d}$ are proven in the proof of Theorem~\ref{lelijk}.
By combining Theorem~\ref{treedw} and Proposition~\ref{gh}, we derive the upper bounds.
\end{proof}

The techniques of this section naturally also apply in the context of the aforementioned conjecture of DeLaVi\~{n}a and Waller. Indeed, this leads to a confirmation of the conjecture for trees. This will appear in a forthcoming work.

\section{An asymptotically extremal upper bound for digraphs}\label{sec4}

In this section we give extremal upper bounds on the Wiener index for digraphs, provided $n$ is large enough.

\begin{thr}\label{digraph_d_1}
Let $d>1$ and $n\ge4d^3$ be integers and $D$ be a digraph of order $n$ with diameter $d$.
Then $W(D)\le dn^2-d^2n+\frac{(d-1)d^2}{2}$.
Equality holds if and only if $D$ is isomorphic to the graph $\DP{n}{d}$ presented in Figure~\ref{fig:digraph_d}, i.e. a directed cycle of length $d$ with a blow-up of one vertex by an independent set of size $n-d+1.$

\end{thr}

\begin{proof}

First, we consider the digraph $\DP{n}{d}$ constructed in Figure~\ref{fig:digraph_d}, which is obtained from a directed path from $u_1$ to $u_{d-1}$ of length $d-1$ and $n-d+1$ vertices $\ell_i$ with directed edges from $u_{d-1}$ to $\ell_i$ and from $\ell_i$ to $u_1$ for every $i$.

Note that $d(\ell_i,\ell_j)=d$ and $d(u_i,u_j)+d(u_j,u_i)=d$ when $i \not= j$ and $d(u_i,\ell_j)+d(\ell_j,u_i)=d$ for every $1 \le i \le d-1$ and $1 \le j \le n-d+1.$
Hence 
\begin{align*}
W(\DP{n}{d}) &= (n-d+1)(n-d)d+(d-1)(n-d+1)d+\frac{(d-1)(d-2)}{2}d\\
&=dn^2-d^2n+\frac{(d-1)d^2}{2}.
\end{align*}

Take an extremal graph $D$ of order $n \ge 4d^3$ and diameter $d.$
For every vertex $u$, we look to the total deficit of $u$ which is the sum $\Delta_u$ of deficits, differences between the maximum possible distances and actual distances between $u$ and the other vertices, $$\Delta_u=\sum_{v \in V: v \not=u} \left( 2d-d(u,v)-d(v,u) \right).$$ Note that for every vertex $u$ we have that $\Delta_u \ge d
^2-d.$ For this note that for all $v$, we have $d(u,v) \le d-1$ or for every $1 \le i \le d-1$ there exists a vertex $w$ such that $d(u,w)=i$. Similarly for $d(v,u).$

If $\Delta_u\ge 2d^2$ for all $u \in V$, then
\begin{align*}
W(D) &= \frac12 \sum_{u \in V} \sum_{v \in V: v \not=u} d(u,v) + d(v,u) \\
&\le \frac12 n\left( 2d(n-1)-2d^2 \right) \\
&=dn^2-(d^2+d)n
\end{align*}
and so this graph was not extremal, since $\DP{n}{d}$ has a larger Wiener index.

Let $u$ be a vertex with $\Delta_u=d^2-d+k$ minimal.
Since for every $1 \le i \le d-1$, there is at least one vertex $v$ such that $d(u,v)=i$ and at least one vertex $w$ such that $d(w,u)=i,$ 
there are at most $2d-2+k$ vertices $v$ such that $\min \{ d(u,v), d(v,u) \} <d$.
Let $A_{u}$ be the set of all vertices $v$ such that $d(u,v)=d(v,u)=d$ and $B_{u}=V \backslash (A_u \cup \{u\}).$
We define the total deficit of $v$ with respect to a subset $S$ as 
$$\Delta_v(S)=\sum_{w \in S: w \not=u} \left( 2d-d(w,v)-d(v,w) \right).$$
Since for every $v \in A_u$, we have $d(u,v)=d(v,u)=d$,
we know that for every $1 \le i \le d-1$ there are vertices $w,z \in B_u$ such that $d(v,w)=i$ and $d(z,v)=i$. This implies that $\Delta_v( B_u) \ge d^2-d.$
We also know $\Delta_v( B_u) + \Delta_v( A_u) \ge d^2-d+k$.

Using all this information, we conclude that if $k\ge 1$ (remember $k<d^2+d$), then 
\begin{align*}
W(D) &\le d(n^2-n)- \sum_{v \in A_u} \Delta_v( B_u) - \frac 12 \sum_{v \in A_u} \Delta_v( A_u)\\
&\le d(n^2-n)- d(d-1) \lvert A_u \rvert  - \frac{k}{2}\lvert A_u \rvert \\
& \le d(n^2-n)-\left(d^2-d+\frac k2 \right)(n-2d-k)\\
&=dn^2-d^2n -\frac{k}{2}n +\left(d^2-d+\frac k2 \right)(2d+k)\\
&< dn^2-d^2n+\frac{d^2(d-1)}{2}
\end{align*}

since 
\begin{align*}
n &\ge 4d^3 \ge \frac{4(d^3-d^2)}{k}+2d^2+k\\
&=\frac{2}{k} \left(d^2-d+\frac k2 \right)(2d+k). 
\end{align*}

So if $k \ge 1$, then $D$ has a smaller Wiener index than the digraph $\DP{n}{d}$ and so it would not be extremal.

If $\Delta_v \ge d^2-d+1$ for all vertices $v \in A_u$, the same calculation with $k=1$ for $W(D)$ holds.
So we conclude some $v \in A_u$ satisfies $\Delta_v=d^2-d$.
Let the unique shortest path $\P$ from $u$ to a fixed $v$ with $\Delta_v=d^2-d$ be $u u_1 u_2 u_3 \ldots u_{d-1} v.$
Since $d(u,w)=d$ and $d(w,v)=d$ for every vertex $w$ not belonging to this path, we have directed edges from $w$ to $u_1$ and from $u_{d-1}$ to $w$ for all such $w$.
If there is an edge from $v$ to some vertex $w$ not on $\P$ or to some vertex $u_i$ with $i>2$, we easily get a contradiction with the fact that for every $1 \le j \le d-1$ there is only one vertex $z$ such that $d(v,z)=j$ and the fact that the graph has diameter $d$.
So $v$ can only be connected to $u_1$.
Similarly, the directed edge ending in $u$ should start in $u_{d-1}$ and the extremal digraph is isomorphic to $\DP{n}{d}$. \end{proof}

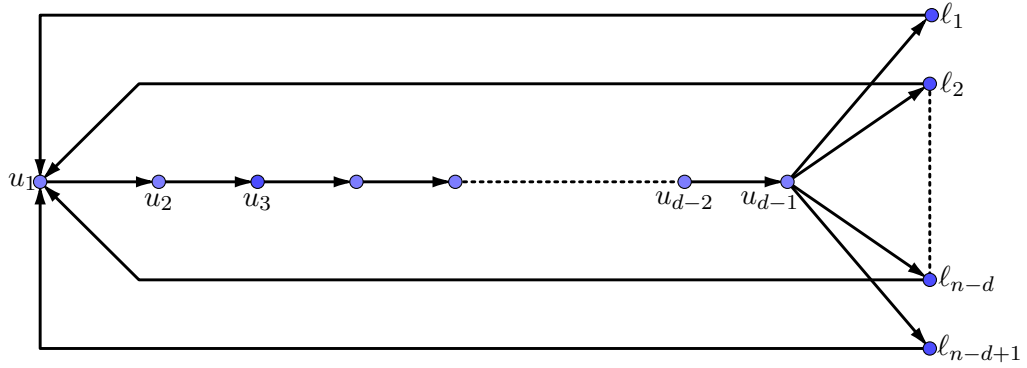
\begin{figure}[h]
\centering
\begin{tikzpicture}[line cap=round,line join=round,>= {Latex[length=3mm, width=1.5mm]} ,x=1.3cm,y=1.3cm]
\clip(-1,-2) rectangle (11,2);
\draw [->,line width=1.1pt] (0.8,0.) -- (2.,0.);
\draw [->,line width=1.1pt] (2.,0.) -- (3.,0.);
\draw [->,line width=1.1pt] (3.,0.) -- (4.,0.);
\draw [->,line width=1.1pt] (4.,0.) -- (5.,0.);
\draw [line width=1.1pt,dotted] (5.,0.)-- (7.32,0.);
\draw [->,line width=1.1pt] (7.32,0.) -- (8.36,0.);
\draw [->,line width=1.1pt] (9.8,1.7) -- (0.8,1.7)--(0.8,0.);
\draw [->,line width=1.1pt] (9.8,1) -- (1.8,1) -- (0.8,0);
\draw [->,line width=1.1pt] (9.8,-1) -- (1.8,-1)-- (0.8,0);
\draw [->,line width=1.1pt] (9.8,-1.7) -- (0.8,-1.7)--(0.8,0.);
\draw [->,line width=1.1pt] (8.36,0.) -- (9.8,-1);
\draw [->,line width=1.1pt] (8.36,0.) -- (9.8,1);
\draw [->,line width=1.1pt] (8.36,0.) -- (9.8,1.7);
\draw [->,line width=1.1pt] (8.36,0.) -- (9.8,-1.7);
\draw [line width=1.1pt,dotted] (9.8,1)-- (9.8,-1);
\draw [fill=xdxdff] (0.8,0.) circle (2.5pt);
\draw [fill=xdxdff] (2.,0.) circle (2.5pt);
\draw [fill=ududff] (3.,0.) circle (2.5pt);
\draw [fill=xdxdff] (4.,0.) circle (2.5pt);
\draw [fill=xdxdff] (5.,0.) circle (2.5pt);
\draw [fill=xdxdff] (7.32,0.) circle (2.5pt);
\draw [fill=xdxdff] (8.36,0.) circle (2.5pt);
\draw [fill=ududff] (9.82,1.7) circle (2.5pt);
\draw [fill=ududff] (9.8,1) circle (2.5pt);
\draw [fill=ududff] (9.8,-1) circle (2.5pt);
\draw [fill=ududff] (9.8,-1.7) circle (2.5pt);
\coordinate [label=center:$u_{2}$] (A) at (2,-0.20); 
\coordinate [label=left:$u_{1}$] (A) at (0.8,0.0); 
\coordinate [label=center:$u_{3}$] (A) at (3,-0.20);
\coordinate [label=left:$u_{d-1}$] (A) at (8.5,-0.2); 
\coordinate [label=center:$u_{d-2}$] (A) at (7.32,-0.20); 

\coordinate [label=right:$\ell_{1}$] (A) at (9.85,1.7); 
\coordinate [label=right:$\ell_{2}$] (A) at (9.85,1); 
\coordinate [label=right:$\ell_{n-d}$] (A) at (9.85,-1); 
\coordinate [label=right:$\ell_{n-d+1}$] (A) at (9.85,-1.7); 

\end{tikzpicture}
\caption{Digraph $\DP{n}{d}$}
\label{fig:digraph_d}
\end{figure}

\section{Conclusion and open problems}

We have determined the magnitude of the second-order term for the largest Wiener index among graphs of order $n$ and diameter $d$.
Since the gap with the trivial upper bound $\frac d2 n^2 $ is of order $d^{3/2}n^{3/2}$, among other reasons, it may be difficult to find sharp, exact upper bounds.

Considering Theorem~\ref{goeieConstantTrees}, the following question could nevertheless be interesting,

\begin{q}
For even $d$ and large $n$, are the graphs of order $n$ and diameter $d$ with largest Wiener index all trees?
\end{q}

In the case of digraphs, we can wonder if we can strengthen the result and give a complete characterization of the extremal graphs.

\begin{q}\label{q:digraph}
For every $d \ge 2$, what is the least value $f(d)$ such that for any $n\ge f(d)$ the digraph $\DP{n}{d}$ is the unique digraph of largest Wiener index over all digraphs of order $n$ and diameter $d$. 
Which graphs are extremal when $d+1<n<f(d)$ or what can be said in that case?
\end{q}

Theorem~\ref{digraph_d_1} showed that $f(d)$ is well-defined.
We note that for $n=d+1$, the directed cycle of order $d+1$ is the unique directed graph attaining the maximal Wiener index.

Furthermore, we show that $f(d)=d+\Omega( \sqrt d)$.

\begin{prop}

The function $f(d)$ in Question~\ref{q:digraph} satisfies $f(d)=d+\Omega( \sqrt d)$. 

\end{prop}

\begin{proof}

For $d+1<n<d+\frac{\sqrt{4d^3+d^2-34d+33}-d+3}{2d-2}$, we construct a graph $\DC{n}{d}$.
We start from a directed cycle $\C=u_0u_1\ldots u_d$ of order $d+1$ and take $k=n-(d+1)$ additional vertices $\ell_i$, with directed edges in both directions between $\ell_i$ and $u_i$ and edges from $u_{i-1}$ to $\ell_i$ and from $\ell_i$ to $u_{i+1}$ for every $1 \le i \le k$.
This digraph has been presented in Figure~\ref{fig:digraph_small n}.

The distances between any two vertices $u_i$ and $u_j$ of $\C$ has not changed by the additional vertices and edges.
For any $i,j$ such that $\lvert i-j \rvert >1$, $d(\ell_i,\ell_j)+d(\ell_j,\ell_i)=d+1$.
When $\lvert i-j \rvert =1$, we have $d(\ell_i,\ell_j)+d(\ell_j,\ell_i)=d+2$, there are $k-1$ such pairs.
Furthermore, $\sum_{u \in \C} d(\ell_i,u)=\sum_{u \in \C} d(u,\ell_i)=1+\frac{d(d+1)}{2}.$

With the previous remarks, we calculate that the Wiener index of $\DC{n}{d}$ equals
\begin{align*}
W(G)&=\frac{d(d+1)^2}{2}+k(d^2+d+2)+\frac{k(k-1)}{2}(d+1)+(k-1)
\end{align*} 

This is larger than the Wiener index of $\DP{n}{d}$ when $$k<\frac{-3d+5+\sqrt{4d^3+d^2-34d+33}}{2d-2}.$$

This implies that $f(d)=d+\Omega( \sqrt d).$
\end{proof}

For $d \in \{2,3\},n=d+2$, the only extremal graphs are $E_n^d$ and $\DP{n}{d}$.
For $d=4,n=6$, there are three extremal graphs, namely $E_6^4,\DP{6}{4}$ and $\DC{6}{4}.$
Here the graphs $E_n^d$ are drawn in Figure~\ref{fig:extremalE}.

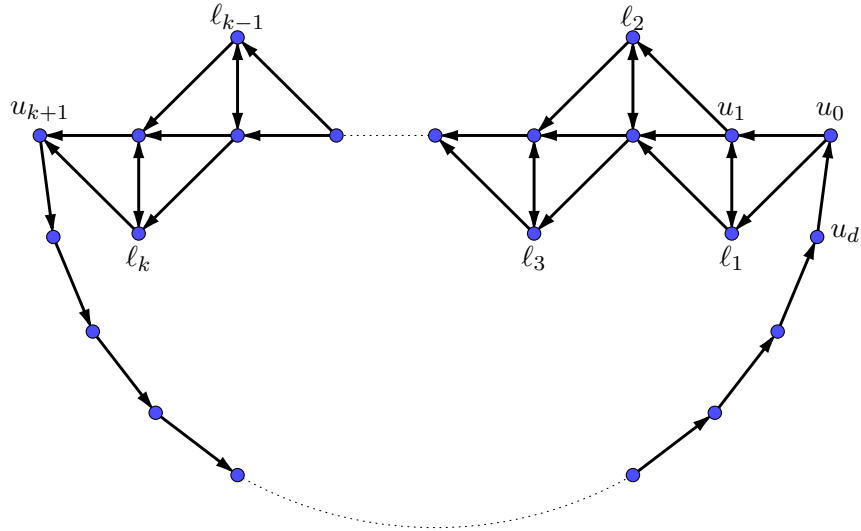
\begin{figure}[h]
\centering
\begin{tikzpicture}
[line cap=round,line join=round,>= {Latex[length=3mm, width=1.5mm]} ,x=1.3cm,y=1.3cm]
\centering [thick,scale=0.4]%
\draw[dotted] (240:4) arc (240:300:4);
\foreach \x in {-3,-2,-1,1,2,3,4} \draw[->,line width=1.1pt] {
	[->,line width=1.1pt]	(\x,0)--  (\x-1,0) };

\foreach \x in {180,195,210,225,300,315,330,345} \draw[->,line width=1.1pt] {
	[->,line width=1.1pt]	(\x:4)--  (\x+15:4) };

\foreach \x in {240,195,210,225,300,315,330,345} \draw[fill=ududff] {
	(\x:4) circle (2.5pt)};

\draw[dotted] (0,0)--(-1,0);

\foreach \x in {-3,1,3}	\draw [<->,line width=1.1pt] {
	(\x,0)-- (\x,-1) };

\foreach \x in {-3,1,3}	\draw [->,line width=1.1pt] {
	(\x+1,0) -- (\x,-1) };

\foreach \x in {-3,1,3}	\draw [->,line width=1.1pt] 
{ (\x,-1) --(\x-1,0)
};

\foreach \x in {-2,2}	\draw [<->,line width=1.1pt] {
	(\x,0)-- (\x,1) };

\foreach \x in {-2,2}	\draw [->,line width=1.1pt] {
	(\x+1,0) -- (\x,1) };

\foreach \x in {-2,2}	\draw [->,line width=1.1pt] 
{ (\x,1) --(\x-1,0)
};

\draw[fill=ududff] (1,-1) circle (2.5pt);
\draw[fill=ududff] (3,-1) circle (2.5pt);
\draw[fill=ududff] (-3,-1) circle (2.5pt);
\draw[fill=ududff] (-2,1) circle (2.5pt);
\draw[fill=ududff] (2,1) circle (2.5pt);
\draw[fill=ududff]  \foreach \x in {-3,-2,-1,1,2,3,4,0,-4} {
	[fill=ududff] (\x,0) circle (2.5pt)};

\coordinate [label=center:$u_{0}$] (A) at (4,0.25); 
\coordinate [label=center:$u_{1}$] (A) at (3,0.25); 
\coordinate [label=center:$u_{k+1}$] (A) at (-4,0.25); 
\coordinate [label=center:$u_{d}$] (A) at (4.15,-1); 

\coordinate [label=center:$\ell_{1}$] (A) at (3,-1.25); 
\coordinate [label=center:$\ell_{2}$] (A) at (2,1.2); 
\coordinate [label=center:$\ell_{3}$] (A) at (1,-1.25); 
\coordinate [label=center:$\ell_{k-1}$] (A) at (-2,1.2); 
\coordinate [label=center:$\ell_{k}$] (A) at (-3,-1.25); 
\end{tikzpicture}

\caption{Digraph $\DC{n}{d}$} 
\label{fig:digraph_small n}
\end{figure}

\begin{figure}[h]
\centering
\begin{tikzpicture}[line cap=round,line join=round,>= {Latex[length=3mm, width=1.5mm]} ,x=1.3cm,y=1.3cm]
\centering [thick,scale=0.4]%
\draw [<->,line width=1.1pt]	(0,0)-- (2,0) ;
\draw [->,line width=1.1pt]	(0,0)-- (1,1.73) ;
\draw [->,line width=1.1pt] (1,1.73)--(2,0);
\draw [->,line width=1.1pt] (2,0)--(1,-1.73);
\draw [->,line width=1.1pt] (1,-1.73)--(0,0) ;

\draw[fill=ududff] (1,1.73) circle (2.5pt);
\draw[fill=ududff] (0,0) circle (2.5pt);
\draw[fill=ududff] (1,-1.73) circle (2.5pt);
\draw[fill=ududff] (2,0) circle (2.5pt);
\end{tikzpicture} \quad
\begin{tikzpicture}[line cap=round,line join=round,>= {Latex[length=3mm, width=1.5mm]} ,x=1.3cm,y=1.3cm]
\centering [thick,scale=0.4]%

\draw [<-,line width=1.1pt]	(0,0)-- (2,0) ;
\draw [<->,line width=1.1pt]	(3,1.73)-- (1,1.73) ;
\draw [->,line width=1.1pt]	(0,0)-- (1,1.73) ;
\draw [->,line width=1.1pt] (1,1.73)--(2,0);
\draw [->,line width=1.1pt] (2,0)--(1,-1.73);
\draw [->,line width=1.1pt] (1,-1.73)--(0,0) ;

\draw[fill=ududff] (1,1.73) circle (2.5pt);
\draw[fill=ududff] (0,0) circle (2.5pt);
\draw[fill=ududff] (1,-1.73) circle (2.5pt);
\draw[fill=ududff] (2,0) circle (2.5pt);
\draw[fill=ududff] (3,1.73) circle (2.5pt);

\end{tikzpicture} 
\quad
\begin{tikzpicture}[line cap=round,line join=round,>= {Latex[length=3mm, width=1.5mm]} ,x=1.3cm,y=1.3cm]
\centering [thick,scale=0.4]%

\draw [<->,line width=1.1pt]	(3,1.73)-- (1,1.73) ;
\draw [<->,line width=1.1pt]	(3,-1.73)-- (1,-1.73) ;
\draw [->,line width=1.1pt]	(0,0)-- (1,1.73) ;
\draw [->,line width=1.1pt] (1,1.73)--(2,0);
\draw [->,line width=1.1pt] (2,0)--(1,-1.73);
\draw [->,line width=1.1pt] (1,-1.73)--(0,0) ;

\draw[fill=ududff] (1,1.73) circle (2.5pt);
\draw[fill=ududff] (0,0) circle (2.5pt);
\draw[fill=ududff] (1,-1.73) circle (2.5pt);
\draw[fill=ududff] (2,0) circle (2.5pt);
\draw[fill=ududff] (3,-1.73) circle (2.5pt);
\draw[fill=ududff] (3,1.73) circle (2.5pt);
\end{tikzpicture}

\caption{Some extremal graphs $E_n^d$ not isomorphic to some $\DP{n}{d}$ or $\DC{n}{d}$}
\label{fig:extremalE}
\end{figure}
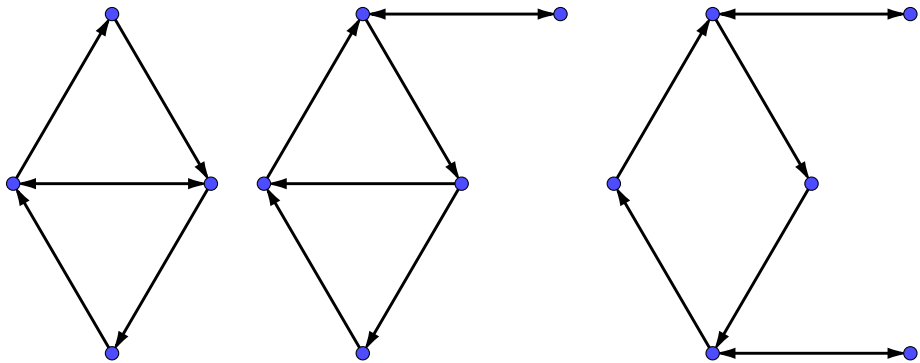

\subsection*{Acknowledgment}
The author is very grateful to Ross Kang and the anonymous referee, for careful reading the manuscript and for several helpful comments.

		\bibliographystyle{abbrv}
		\bibliography{MaxMu}

	\end{document}